\documentclass[titlepage,oneside,letterpage,12pt]{article}

\usepackage[tiny,rm]{titlesec}
\newpagestyle{trbstyle}{
\sethead{Jin and Amin}{}{\thepage}
}
\titleformat{\section}{\bfseries}{\thesection}{12pt}{\uppercase}
\titlespacing*{\section}{0pt}{12pt}{*0}
\titleformat{\subsection}{\bfseries}{\thesubsection}{12pt}{}
\titlespacing*{\subsection}{0pt}{12pt}{*0}
\titleformat{\subsubsection}{\itshape}{}{0pt}{}
\titlespacing*{\subsubsection}{0pt}{12pt}{*0}

\usepackage{enumitem}
\setlist[1]{labelindent=0.5in,leftmargin=*}
\setlist[2]{labelindent=0in,leftmargin=*}
\setlist{nosep} 

\usepackage{ccaption}
\usepackage{amsmath}
\makeatletter
\renewcommand{\fnum@figure}{\textbf{FIGURE~\thefigure} }
\renewcommand{\fnum@table}{\textbf{TABLE~\thetable} }
\makeatother
\captiontitlefont{\bfseries \boldmath}
\captiondelim{\;}

\usepackage{mathptmx}

\usepackage[T1]{fontenc}
\usepackage{textcomp}
\usepackage[noae]{Sweave}

\usepackage[sort&compress,numbers]{natbib}
\renewcommand{\cite}[1]{({\it \citenum{#1}})}

\setcitestyle{round}

\newread\somefile
\usepackage{xparse}

\newcounter{totalwordcounter}
\newcounter{wordcounter}
\makeatletter

\NewDocumentCommand{\wordcount}{s}{%
  \immediate\write18{texcount -sum -1 \jobname.tex > count.txt}%
  \immediate\openin\somefile=count.txt%
  \read\somefile to \@@localdummy%
  \immediate\closein\somefile%
  \setcounter{wordcounter}{\@@localdummy}%
  \IfBooleanF{#1}{%
  \@@localdummy
  }%
}
\makeatother

\usepackage{totcount}
	\regtotcounter{table} 	
	\regtotcounter{figure} 	

\newcommand{\wordfigure}{250} 
\newcommand{\wordtable}{250} 

\newcommand{\totalwordcount}{%
  \wordcount*
  \setcounter{totalwordcounter}{\value{wordcounter}}%
  \addtocounter{totalwordcounter}{\numexpr\wordfigure*\totvalue{figure}}%
  \addtocounter{totalwordcounter}{\numexpr\wordtable*\totvalue{table}} %
  \number\value{totalwordcounter}
  \renewcommand{\totalwordcount}{\number\value{totalwordcounter}}
}

\usepackage{graphicx}
\usepackage{algorithm}
\usepackage{algpseudocode}
\makeatletter
\def\BState{\State\hskip-\ALG@thistlm}
\makeatother
\usepackage{amsmath}
\usepackage{amsthm} 
\usepackage{amssymb}
\usepackage{amsbsy}
\usepackage{pbox}
\usepackage{subfigure}
\usepackage{url}
\usepackage{epstopdf}

\usepackage{Sweave}
\begin{document}
\Sconcordance{concordance:trb_template.tex:trb_template.Rnw:%
1 371 1}

\thispagestyle{empty}

\newtheorem{dfn}{Definition}
\newtheorem{thm}{Theorem}
\newtheorem{exm}{Example}
\newtheorem{prp}{Proposition}
\newtheorem{rem}{Remark}
\newtheorem{lmm}{Lemma}
\newtheorem{asm}{Assumption}
\newtheorem{clm}{Claim}
\newtheorem{cor}{Corollary}
\newtheorem{cnd}{Condition}
%
\newcommand{\Dens}{\rho} 
\newcommand{\dens}{\rho} 
\newcommand{\Densn}{n} 
\newcommand{\densn}{n} 
\newcommand{\jmdens}{\bar{\rho}} 
\newcommand{\jmdensn}{\bar{n}} 
\newcommand{\ffspd}{v} 
\newcommand{\flwspd}{u} 
\newcommand{\conspd}{w} 
\newcommand{\flow}{\mathsf{f}} 
\newcommand{\Flow}{f} 
\newcommand{\flowfn}{G} 
\newcommand{\Flowfn}{\Theta} 
\newcommand{\redfac}{\alpha} 
\newcommand{\Redfac}{\mathcal{A}} 
\newcommand{\splt}{\beta} 
\newcommand{\capact}{\mathsf{F}} 
\newcommand{\sfF}{\mathsf{F}} 
\newcommand{\sfS}{\mathsf{S}} 
\newcommand{\sfR}{\mathsf{R}} 
\newcommand{\Capact}{\mathcal{F}} 
\newcommand{\oricap}{\bar{f}^0} 
\newcommand{\inflow}{r} 
\newcommand{\Inflow}{r} 
\newcommand{\outflw}{s} 
\newcommand{\sent}{S} 
\newcommand{\recv}{R} 
\newcommand{\maxflw}{M} 
\newcommand{\ttt}{\mathsf{VHT}} 
\newcommand{\vmt}{\mathsf{VMT}} 
\newcommand{\vht}{\mathsf{VHT}} 
\newcommand{\Densrv}{\rho} 
\newcommand{\CS}{X} 
\newcommand{\ctmdim}{N} 
\newcommand{\fftn}{\hat{f}} 
\newcommand{\eqflow}{f} 
\newcommand{\Eqflow}{f} 
\newcommand{\Eqdens}{z} 
\newcommand{\eqdens}{z} 
\newcommand{\PL}{\mathsf{PL}} 
\newcommand{\pl}{\mathsf{pl}} 
\newcommand{\eqspd}{u^*} 
\newcommand{\nol}{z} 
\newcommand{\loc}{l} 
\newcommand{\denied}{\Delta} 
\newcommand{\botind}{I} 
\newcommand{\incrt}{\mu} 

\newcommand{\shs}{\{X(t),t\ge0\}} 
\newcommand{\pdmp}{\{(\Dismod(t),\Densrv(t));\;t\ge0\}} 
\newcommand{\pdmpZ}{\{Z_t;\;t\ge0\}} 
\newcommand{\chain}{\{\tilde Z_n,n\in\mathbb N_+\}} 
\newcommand{\hybrsp}{\mathscr{E}} 
\newcommand{\contsp}{\Rho} 
\newcommand{\condom}{\mathcal{E}} 
\newcommand{\disdom}{\mathcal{E}} 
\newcommand{\jumprt}{\lambda} 
\newcommand{\tranrt}{q} 
\newcommand{\clrrt}{\mu} 
\newcommand{\incmod}{\delta} 
\newcommand{\Incmod}{\Delta} 
\newcommand{\dimim}{|M|} 
\newcommand{\Incmrv}{\sigma} 
\newcommand{\incmsp}{M} 
\newcommand{\conmod}{k} 
\newcommand{\dimcm}{K^j} 
\newcommand{\Conmd}{\mathcal K} 
\newcommand{\Conmod}{\mathcal K^{m}} 
\newcommand{\deccon}{\tilde k} 
\newcommand{\Deccon}{\tilde{K}} 
\newcommand{\dimdc}{K^*} 
\newcommand{\CDF}{\mathsf{F}} 
\newcommand{\transi}{\mathcal{P}} 
\newcommand{\vecfld}{\mathscr{X}} 
\newcommand{\jmptim}{T} 
\newcommand{\tratim}{\tilde T} 
\newcommand{\jtsp}{t} 
\newcommand{\epoch}{S} 
\newcommand{\frstjt}{T_1} 
\newcommand{\Jmptim}{\mathcal{T}} 
\newcommand{\orbit}{\phi} 
\newcommand{\trajpc}{\phi} 
\newcommand{\traj}{\phi} 
\newcommand{\trajcc}{\Phi} 
\newcommand{\incide}{I} 
\newcommand{\ftt}{t^*} 
\newcommand{\intcrv}{\phi} 
\newcommand{\dismod}{\nu} 
\newcommand{\Dismod}{V} 
\newcommand{\dimdm}{\mathcal N} 
\newcommand{\discsp}{\mathcal{V}} 
\newcommand{\bndr}{\partial\mathcal{E}} 
\newcommand{\dmseq}{\boldsymbol{\nu}} 
\newcommand{\tranms}{\mathscr{Q}} 
\newcommand{\tranpr}{\mathcal{Q}} 
\newcommand{\genrtr}{\mathscr{A}} 
\newcommand{\matA}{\mathcal{A}} 
\newcommand{\matB}{\mathcal{B}} 
\newcommand{\matC}{\mathcal{C}} 
\newcommand{\Bndr}{\Gamma} 
\newcommand{\redcon}{\mathsf{a}} 
\newcommand{\rfset}{\mathcal{A}} 
\newcommand{\trset}{\Lambda} 
\newcommand{\spind}{n} 
\newcommand{\ASet}{\Gamma} 
\newcommand{\ittsp}{t} 
\newcommand{\epsp}{t} 
\newcommand{\Tseq}{\mathbb{T}} 
\newcommand{\tseq}{\mathcal{t}} 
\newcommand{\Jumprt}{\Lambda} 
\newcommand{\pstvtr}{\gamma^+} 
\newcommand{\hyprec}{\mathcal{H}} 
\newcommand{\imtup}{(\lambda,\alpha)} 
\newcommand{\Imtup}{(\Lambda,A)} 
\newcommand{\borel}{\mathcal{B}} 
\newcommand{\dk}{\mathcal{H}} 
\newcommand{\sk}{\mathcal{J}} 
\newcommand{\mk}{\mathcal{G}} 
\newcommand{\CSD}{\Theta} 
\newcommand{\hyb}{Z} 
\newcommand{\hybt}{\tilde Z} 
\newcommand{\Denst}{\tilde X} 
\newcommand{\incmdt}{\tilde m} 
\newcommand{\semgrp}{P_t} 
\newcommand{\semgrt}{\tilde P} 
\newcommand{\Kt}{K_t} 
\newcommand{\Ktld}{\tilde K} 

\newcommand{\real}{\mathbb{R}} 
\newcommand{\natur}{\mathbb{N}} 
\newcommand{\integr}{\mathbb{Z}_+} 
\newcommand{\probab}{\mathsf{Pr}} 
\newcommand{\E}{\mathsf{E}} 
\newcommand{\sampsp}{\Omega} 
\newcommand{\samppt}{\omega} 
\newcommand{\ptran}{P} 
\newcommand{\ftn}{g} 
\newcommand{\stomat}{\mathcal{P}} 
\newcommand{\simtim}{t_{\mbox{sim}}} 
\newcommand{\indica}{\mathbb{I}} 
\newcommand{\sfi}{\mathsf{i}} 
\newcommand{\sfe}{\mathsf{e}} 
\newcommand{\sfj}{\mathsf{j}} 
\newcommand{\sfk}{\mathsf{k}} 
\newcommand{\dt}{\delta} 
\newcommand{\invms}{\mu} 
\newcommand{\ones}{\mathcal{e}} 
\newcommand{\ys}{\xi} 
\newcommand{\ns}{n^2} 
\newcommand{\resil}{R} 
\newcommand{\degra}{D} 
\newcommand{\cost}{C} 
\newcommand{\argmin}{\mathrm{arg}\min} 
\newcommand{\kernl}{\mathscr{K}_t} 
\newcommand{\kernt}{\tilde {\mathscr K}} 
\newcommand{\prbms}{\mathcal{P}} 
\newcommand{\conftn}{\mathcal{C}_0} 
\newcommand{\ball}{\mathcal{B}}
\newcommand{\supp}{\mathsf{supp}}
\newcommand{\la}{\langle}
\newcommand{\ra}{\rangle}
\newcommand{\lb}{\left(}
\newcommand{\rb}{\right)}
\newcommand{\lc}{\left\{}
\newcommand{\rc}{\right\}}
\newcommand{\nom}{0}
\newcommand{\unc}{{\mathrm{unc}}}
\newcommand{\con}{{\mathrm{con}}}


\newcommand{\bH}{\mathcal{H}}
\newcommand{\I}{\mathcal{I}}
\newcommand{\T}{\mathcal{T}}
\newcommand{\bP}{\mathcal{P}}
\newcommand{\Q}{\mathcal{Q}}
\newcommand{\Ou}{\mathcal{O}}
\newcommand{\W}{\mathcal{W}}
\newcommand{\X}{\mathcal{X}}
\newcommand{\x}{\mathcal{x}}
\newcommand{\y}{\mathcal{y}}
\newcommand{\Y}{\mathcal{Y}}
\newcommand{\Z}{\mathbb{Z}}
\newcommand{\A}{\mathcal{A}}
\newcommand{\B}{\mathcal{B}}
\newcommand{\C}{\mathcal{C}}
\newcommand{\D}{\mathcal{D}}
\newcommand{\Bbar}{\overline{\mathcal{B}}}
\newcommand{\N}{\mathcal{N}}
\newcommand{\R}{\mathcal{R}}
\newcommand{\bd}{\mathcal{d}}
\newcommand{\ubar}{\overline u}
\newcommand{\bfu}{\mathbf{u}}
\newcommand{\bs}{\mathcal{s}}
\newcommand{\br}{\mathcal{r}}
\newcommand{\bu}{\mathcal{u}}
\newcommand{\bp}{\mathcal{p}}
\newcommand{\bh}{\mathcal{h}}
\newcommand{\bff}{\mathbf{f}}
\newcommand{\e}{\mathsf{e}}
\newcommand{\ej}{\mathsf{e}}
\newcommand{\bc}{\mathcal{c}}
\newcommand{\bb}{\mathcal{b}}
\newcommand{\ba}{\mathcal{a}}
\newcommand{\bbeta}{\boldsymbol{\beta}}
\newcommand{\0}{\mathcal{0}}
\newcommand{\Polyh}{\mathcal P}
\newcommand{\bfs}{\tilde{\mathcal x}}
\newcommand{\st}{\mathrm{s.t.}}
\newcommand{\edge}{\mathcal E}
\newcommand{\node}{\mathcal N}
\newcommand{\graph}{\mathcal G}
\newcommand{\Lagr}{\mathcal L}
\newcommand{\Jbar}{{\overline J}}
\newcommand{\Qbar}{{\overline Q}}
\newcommand{\Rho}{{\mathrm P }}
\newcommand{\U}{{\mathcal U}}
\newcommand{\F}{{\mathcal F}}
\newcommand{\fbar}{{\overline f}}
\newcommand{\rbar}{{\overline r}}
\newcommand{\cbar}{\overline c}
\newcommand{\dbar}{\overline d}
\newcommand{\qbar}{\overline q}
\newcommand{\sbar}{\overline s}
\newcommand{\tbar}{\overline t}
\newcommand{\wbar}{\overline w}
\newcommand{\sfp}{\mathsf{p}}
\newcommand{\sfr}{\mathsf{r}}
\newcommand{\sfs}{\mathsf{s}}
\newcommand{\sft}{\mathsf{t}}
\newcommand{\rhobar}{\overline\rho}
\newcommand{\Beta}{\mathrm{B}}
\newcommand{\Alpha}{\mathrm{A}}
\newcommand{\diag}{\mathrm{diag}}
\newcommand{\Ell}{\mathcal{L}}
\newcommand{\Mu}{\mathrm{M}}
\newcommand{\net}{\mathcal{T}}
\newcommand{\cell}{\mathcal{C}}
\newcommand{\arc}{\mathcal{A}}
\newcommand{\into}{{\mathrm{in}}}
\newcommand{\out}{{\mathrm{out}}}
\newcommand{\bfF}{\mathbf{F}}
\newcommand{\bfE}{\mathbf{E}}
\newcommand{\bfD}{\mathbf{D}}
\newcommand{\bfI}{\mathbf{I}}
\newcommand{\bfLmd}{\boldsymbol{\Lambda}}
\newcommand{\bfg}{\mathbf{g}}
\newcommand{\bfd}{\mathbf{d}}
\newcommand{\bfb}{\mathbf{b}}
\newcommand{\bfr}{\mathbf{r}}
\newcommand{\bfe}{\mathbf{e}}
\newcommand{\bfk}{\mathbf{k}}
\newcommand{\bfv}{\mathbf{v}}
\newcommand{\bfq}{\mathbf{q}}
\newcommand{\bfp}{\mathbf{p}}
\newcommand{\bfU}{\mathbf{U}}
\newcommand{\bfR}{\mathbf{R}}
\newcommand{\bfV}{\mathbf{V}}
\newcommand{\bfJ}{\mathbf{J}}
\newcommand{\bfA}{\mathbf{A}}
\newcommand{\bfB}{\mathbf{B}}
\newcommand{\bfS}{\mathbf{S}}
\newcommand{\bfpi}{\boldsymbol{\pi}}
\newcommand{\bfPhi}{\boldsymbol{\Phi}}
\newcommand{\bfphi}{\boldsymbol{\phi}}
\newcommand{\bfGmm}{\boldsymbol{\Gamma}}
\newcommand{\sfz}{\mathsf{z}}
\newcommand{\bfz}{\mathbf{z}}
\newcommand{\bfx}{\mathbf{x}}
\newcommand{\bfy}{\mathbf{y}}
\newcommand{\sff}{\mathsf{f}}
\newcommand{\sgn}{\mathrm{sgn}}
\newcommand{\Init}{\I\times\Q}
\newcommand{\bfpsi}{\boldsymbol{\psi}}
\newcommand{\bfPsi}{\boldsymbol{\Psi}}
\newcommand{\bfTht}{\boldsymbol{\Theta}}
\newcommand{\bfEta}{\mathbf{H}}
\newcommand{\Eta}{\mathrm{H}}
\newcommand{\K}{\mathcal{K}}


\begin{titlepage}
\begin{flushleft}

{\MakeUppercase{\bfseries Analyzing a Tandem Fluid Queueing Model with Stochastic Capacity and Spillback}}\\[36pt]

{\bfseries Li Jin} \\
Tandon School of Engineering\\
New York University\\
15 MetroTech Center \\
Brooklyn, NY 11201 \\
lijin@nyu.edu\\[12pt]

{\bfseries Saurabh Amin}\\
Department of Civil and Environmental Engineering\\
Massachusetts Institute of Technology \\
77 Massachusetts Avenue \\
Cambridge, MA 02139 \\
amins@mit.edu\\[60pt]

Word Count: \wordcount words + \total{figure} figure(s) + \total{table} table(s) = \totalwordcount~words\\[12pt]

Submission Date: \today
\end{flushleft}
\end{titlepage}

\newpage
\newpage
\section*{Abstract}
The tandem fluid queueing model is a useful tool for performance analysis and control design for a variety of transportation systems.
In this article, we study the joint impact of stochastic capacity and spillback on the long-time properties of this model.
Our analysis focuses on the system of two fluid queueing links in series. 
The upstream link has a constant capacity (saturation rate) and an infinite buffer size.
The downstream link has a stochastic capacity and a finite buffer size.
Queue spillback occurs when the the downstream link is full.
We derive a necessary condition and a sufficient condition for the total queue length to be bounded on average.
The necessary (resp. sufficient) condition leads to an upper (resp. lower) bound for the throughput of the two-link system.
Using our results, we analyze the sensitivity of throughput of the two-link system with respect to the frequency and intensity of capacity disruptions, and to the buffer size.
In addition, we discuss how our analysis can be extended to feedback-controlled systems and to networks consisting of merges and splits.\\

\noindent\emph{Keywords}: 
Fluid queueing model, queue spillback, stability analysis, stochastic capacity, throughput.
\newpage
\section{Introduction}

Capacity disruptions are common in transportation systems.
Typical examples include incidents on freeways \cite{skabardonis97,schrank12} and weather-related capacity drops at airports \cite{peterson95,ny11}.
In addition, some authors pointed out that the performance of transportation systems is also affected by congestion propagation within the system, i.e. spillback \cite{daganzo98,papageorgiou03}. 
In this article, we study the behavior of transportation systems under the effect of both capacity fluctuation and spillback.

Our study is based on the tandem fluid queueing model, a useful tool for performance analysis and control design for a variety of transportation systems, including highway systems \cite{newell13,jin2018modeling} and air transportation \cite{odoni87,bertsimas00}.
The impact of stochastic capacity and that of spillback have been studied by two relatively independent lines of work.
Both lines of work are based on fluid queueing models.
The first line of work \cite{anick82,kulkarni97,jin16} focuses on the stability and steady-state queue length of individual fluid queueing links with stochastically varying capacities.
The second line of work \cite{daganzo98,ran94,shen14} studies the equilibrium flow of deterministic queueing networks with spillback. 
However, very limited results are available for servers with both time-varying service rates and spillback.
In this article, we use a simple fluid queueing model to study the {\bf joint} impact of these two factors, which provides new insights for operations of transportation systems.

Specifically, we consider two fluid queueing links in series, where a constant inflow is sent to the upstream link.
The upstream link has a constant capacity and an infinite buffer size.
The downstream link has a stochastically varying capacity and a finite buffer size. Queue spillback happens when the queue in the downstream link attains the buffer size.
Note that similar models can also capture fluctuations in the arrival process of vehicles \cite{newell13} or the impact of vehicle platoons \cite{jin2018modeling}.
Following \cite{dai95ii}, we consider the two-link system to be stable if the total queue length is bounded on average, i.e. the long-time average of the total queue length being bounded.
We view the supremum of the set of stable inflows as the throughput.
Our objective is to derive stability conditions for the two-link system, which leads to bounds on the throughput.
Although we are motivated by transportation applications, our approach are also relevant for communication networks and manufacturing systems \cite{kulkarni97,yu04}.


The main results of this article (Propositions~\ref{Prp_Stable1} and \ref{Prp_Stable2}) provide stability conditions for the two-link system.
The stability conditions build on known results on stability analysis of continuous-time Markov processes \cite{davis84,meyn93} and steady-state behavior of stochastic fluid queueing models \cite{anick82,kulkarni97,jin16}.

First, a necessary condition (Propositions~\ref{Prp_Stable1}) is derived based on known results on the steady-state behavior of single fluid queueing links with finite buffer sizes \cite{anick82,kulkarni97}.
This condition estimates the actual throughput by incorporating an estimate of the spillback probability.
As a necessary condition, Propositions~\ref{Prp_Stable1} provides an upper bound for the throughput.
An important insight from this result is that the two-link system is not necessarily stable even if the inflow is strictly less than the time-average capacity of each link.
Thus, knowing the average capacity is not necessarily sufficient for the purpose of efficient operations of transportation systems.

Second, a sufficient condition (Propositions~\ref{Prp_Stable2}) is that the inflow and the parameters verify a set of linear inequalities.
To derive this condition, we consider a polynomial Lyapunov function for the two-link system and apply the Foster-Lyapunov drift condition \cite{meyn93}.
As a sufficient condition, Propositions~\ref{Prp_Stable2} provides a lower bound for the throughput.

Using the above results, we analyze how the throughput of the two-link system varies with the magnitude of capacity fluctuation and the buffer size. Our throughput analysis implies the following conclusions. 
First, throughput decreases with capacity variation and increases with buffer size.
Second, a small number of major capacity disruptions lead to more throughput loss than a large number of minor disruptions.
Third, throughput is the most sensitive to capacity variation.

Furthermore, we discuss two directions in which the analysis of the two-link system can be extended. First, we argue that our approach can be used to analyze fluid queueing systems with a class of feedback control policies. Second, we discuss how our analysis can be extended to more general networks. Specifically, we argue the extension of our results to merges and splits, which are the basic structures of general networks.

The rest of this article are organized as follows. In Section~\ref{Sec_Model}, we introduce the two-link system model. In Sections~\ref{Sec_Stable}, we derive the necessary condition and the sufficient condition, respectively. In Section~\ref{Sec_Analysis}, we analyze the throughput of the two-link system.
In Section~\ref{Sec_Discussion}, we discuss possible extensions of our results.
Section~\ref{Sec_Conclude} gives the concluding remarks.
\newpage
\section{System model}
\label{Sec_Model}

In this section, we define the two-link fluid queueing model, and introduce the assumptions that we use in our analysis.

\begin{figure}[hbt]
\centering
\includegraphics[width=0.3\textwidth]{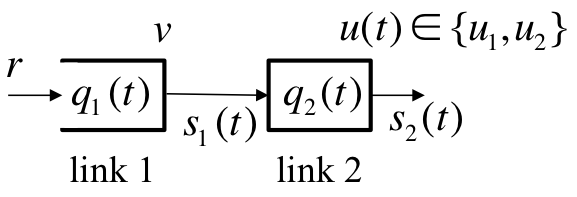}
\caption{A two-link fluid queueing model.}
\label{Fig_PDQ}
\end{figure}

Consider the system in Figure~\ref{Fig_PDQ}, which consists of two links in series.
Traffic arrives at link 1 (the upstream link) at the constant \emph{inflow rate} $r$. 
Assume that $r$ can take value in $\R=\real_{\ge0}$.
The outflow from link 1 goes to link 2.
Let the vector of \emph{queue lengths} be denoted by $q(t)=[q_1(t),q_2(t)]^T$.
We assume that link 1 has an infinite buffer size (i.e. $q_1\in[0,\infty)$) and link 2 has a finite buffer size of $\theta$ (i.e. $q_2\in[0,\theta]$). Thus, the vector $q$ can take value in $\Q:=[0,\infty)\times[0,\theta]$.

Let $v$ denote the \emph{capacity} of link 1, i.e. the maximum rate at which $q_1$ can be discharged.
Assume for simplicity that link 1 has a constant capacity $v$.
However, let $u(t)$ denote the capacity of link 2, which switches between two values $u_1$ and $u_2$.
Our analysis in this article focuses on the following case:
\begin{asm}
\label{Asm_uv}
$0\le u_2\le v\le u_1$.
\end{asm}
Note that $u_2\le u_1$ is without loss of generality. Regarding $v$, if $v<u_2$, then no queue exists in link 2 after sufficiently long time; if $v>u_1$, then link 2 is the only bottleneck of the two-link system, and the system would behavior just like a single link with capacity $u(t)$; only under the Assumption~\ref{Asm_uv} is the interaction between the two links of interest.

Now we specify how $u(t)$ varies with time.
Let $\I=\{1,2\}$ be the set of \emph{modes} of the two-link system.
We denote the mode at time $t$ by $i(t)$. Each mode $i\in\I$ is associated with a \emph{fixed} capacity, denoted by $u_i$.
The evolution of mode $i(t)$ is governed by a two-state Markov process; the \emph{transition rate} from mode 1 to mode 2 is $\lambda>0$, and the transition rate from mode 2 to mode 1 is $\mu>0$. 

Given an initial mode $i_0\in\I$ at $t=t_0=0$, let $\{t_k;k=1,2,\ldots\}$ be the \emph{epochs} at which the mode transitions occur. Let $i_{k-1}$ be the mode during $[t_{k-1},t_k)$ and $x_k=t_k-t_{k-1}$. Then, $x_k$ follows an exponential distribution with the cumulative distribution function (CDF) \cite{gallager13}:
\begin{align}
\mathsf F_{x_k}(x)=\left\{\begin{array}{ll}
1-\exp({-\lambda x}), & i_{k-1}=1,\\
1-\exp({-\mu x}), & i_{k-1}=2,
\end{array}\right.
\quad k=1,2\ldots
\label{Eq_Exp}
\end{align}
We can write the transition rates in the $2\times 2$ matrix:
\begin{align}
\Lambda:=\left[\begin{array}{cc}
-\lambda & \lambda\\
\mu & -\mu
\end{array}\right].
\label{Eq_Lmd}
\end{align}

By Theorem 7.2.7 in \cite{gallager13}, the mode transition process $\{i(t);t\ge0\}$ converges towards a unique steady-state distribution, i.e. a row vector $\sfp=[\sfp_1,\sfp_2]$ satisfying
\begin{align}
\sfp\Lambda=0,\;|\sfp|=1,\;\sfp\ge0,
\label{Eq_pLmd}
\end{align}
where $|\sfp|=\sfp_1+\sfp_2$. One can easily see that
\begin{align*}
\sfp_1=\frac{\mu}{\lambda+\mu},\;
\sfp_2=\frac{\lambda}{\lambda+\mu}.
\end{align*}

The \emph{discharge rates} $s$ of both links can be written as functions of the mode $i$, the vector of queue lengths $q$, and the inflow $r$:
\begin{subequations}
\begin{align}
&s_1(i,q,r):=\left\{\begin{array}{ll}
r, & q_1=0,\; q_2<\theta,\;r\le v,\\
v, & q_1>0,\;q_2<\theta,\\
v, & q_1>0,\;q_2=\theta,\;v\le u_i,\\
u_i, & o.w.
\end{array}\right.
\label{Eq_s1}\\
&s_2(i,q,r):=\left\{\begin{array}{ll}
r, & q_1=q_2=0,\; r\le \min\{v, u_i\},\\
v, & q_1=q_2=0,\;v\le \min\{r, u_i\},\\
v, & q_1>0,\;q_2=0,\; v\le u_i,\\
u_i, & o.w.
\end{array}\right.
\label{Eq_s2}
\end{align}
\label{Eq_s}
\end{subequations}
\hspace{-5pt}Note that \eqref{Eq_s1} accounts for the effect of spillback.
Then, we define a vector field $F:\I\times\Q\times\R\to\real^2$ as follows:
\begin{align}
F(i,q,r):=\left[\begin{array}{c}
r-s_1(i,q,r)\\
s_1(i,q,r)-s_2(i,q,r)
\end{array}\right].
\label{Eq_D}
\end{align}

Thus, the evolution of the \emph{hybrid state} $(i(t),q(t))$ of the two-link system is specified by the matrix $\Lambda$ and the vector field $F$ as follows
\begin{subequations}
\begin{align}
&i(0)=i,\;q(0)=q,\quad (i,q)\in\I\times\Q,\label{Eq_q0}\\
&\Pr\{i(t+\Delta t)=j|i(t)=i\}=(\lambda\mathbf1_{i=1}+\mu\mathbf1_{i=2})\Delta t+\mathrm o(\Delta t),\\
&\frac{dq(t)}{dt}=F\Big(i(t),q(t),r\Big).
\label{Eq_D}
\end{align}
\label{Eq_System}
\end{subequations}

The system defined in \eqref{Eq_System} is in fact a \emph{piecewise-deterministic Markov process} (PDMP, see \cite{davis84,benaim15}).
One can easily check that, for any initial condition $(i,q)\in\Init$, the integral curve induced by the vector field $F(i,q,r)$ is unique and continuous. Furthermore, $q(t)$ is not reset after mode transitions. Thus, the stochastic process $\{(i(t),q(t));t\ge0\}$ is a right continuous with left limits (RCLL, or \emph{c\`adl\`ag}) PDMP \cite{davis84}.
Then, following Theorem 5.5 in \cite{davis84}, the \emph{infinitesimal generator} $\Ell$ of the two-link system with inflow $r\in\R$ is given by
\begin{align}
\Ell g(1,q)&=F(1,q,r)\frac{\partial g(1,q)}{\partial q}+\lambda\Big( g(2,q)-g(1,q)\Big),\;q\in\Q\nonumber\\
\Ell g(2,q)&=F(2,q,r)\frac{\partial g(2,q)}{\partial q}+\mu\Big( g(1,q)-g(2,q)\Big),\;q\in\Q,
\label{Eq_Lg}
\end{align}
where $g$ is any function on $\I\times\Q$ smooth in the continuous argument.

In this article, we follow \cite{dai95} and consider the following notion of stability. We say that the total queue length is \emph{bounded on average} if there exists $K<\infty$ such that, for each initial condition $(i,q)\in\I\times\Q$,
\begin{align}
\limsup_{t\to\infty}\frac1t\int_{\tau=0}^t\E[|q(\tau)|]d\tau\le K,
\label{Eq_Bounded}
\end{align}
where $|q|=q_1+q_2$.
In addition, we say that the two-link system is \emph{non-evanescent} if, for each initial condition $(i,q)\in\I\times\Q$,
\begin{align}
\Pr\left\{\lim_{t\to\infty}|q(t)|=\infty\Big|i(0)=i,q(0)=q\right\}=0;
\label{Eq_NonEvanescence}
\end{align}
i.e., the system is non-evanescent if the queue length is finite almost surely (a.s.).
According to \cite{meyn93}, non-evanescence is a necessary condition for boundedness on average.

Finally, given a two-link system, we define the \emph{maximum throughput} of the two-link system, denoted by $J_{\max}$, as the supremum of the set of inflows $r$ such that the first moment of the vector of queue lengths is bounded on average in the sense of \eqref{Eq_Bounded}.
\newpage
\section{Stability conditions for two-link system}
\label{Sec_Stable}

In this section, we derive a necessary condition (Section~\ref{Sec_Stable1}) and a sufficient condition (Section~\ref{Sec_Stable2}) for the stability of the two-link system. These results are basis for our subsequent analysis.

\subsection{Necessary condition: Non-evanescence}
\label{Sec_Stable1}

The necessary condition is based on known results of the steady-state behavior of fluid queueing models with stochastically switching capacities and finite buffer sizes \cite{kulkarni97}.
For the sake of completeness, we recall here the results from \cite{kulkarni97}. Consider a constant $r\ge0$ such that
\begin{align}
r<\min\left\{v,\frac{\mu}{\lambda+\mu}u_1+\frac{\lambda}{\lambda+\mu}u_2\right\}.
\label{Eq_Stable0}
\end{align}
One will see that this constant is essentially the inflow (and hence we use the same notation).
Define
\begin{align*}
D:=\left[\begin{array}{cc}
r-u_1 & 0\\
0 & r-u_2
\end{array}\right]
\end{align*}
and let $\Lambda$ as defined in \eqref{Eq_Lmd}.
Let $w_1$ and $w_2$ be the distinct solutions to the equation
\begin{align*}
\det[wD-\Lambda]=0.
\end{align*}
Let row vectors $\phi_1=[\phi_{11},\phi_{12}]$ and $\phi_2=[\phi_{21},\phi_{22}]$ and scalars $k_1$, $k_2$ be the solutions to
\begin{align*}
&\phi_1[w_1D-\Lambda]=0,\\
&\phi_2[w_2D-\Lambda]=0,\\
&k_1\phi_{12}+k_2\phi_{22}=0,\\
&k_1\phi_{11}\exp(w_1\theta)+k_2\phi_{21}\exp(w_2\theta)=\frac{\mu}{\lambda+\mu}.
\end{align*}
Next, define
\begin{align}
\hat\sfp:=\frac{\lambda}{\lambda+\mu}-k_1\phi_{12}\exp(w_1\theta)-k_2\phi_{22}\exp(w_2\theta).
\label{Eq_phat}
\end{align}

Then, recalling the definition of $v$, we state our necessary condition as follows:

\begin{prp}
\label{Prp_Stable1}
Consider the two-link system defined in \eqref{Eq_System}.
If the total queue lengths are bounded on average in the sense of \eqref{Eq_Bounded}, then
\begin{align}
r\le(1-\hat\sfp)v+\hat\sfp u_2,
\label{Eq_Stable1}
\end{align}
where $\hat\sfp$ is defined in \eqref{Eq_phat}.
\end{prp}

The necessary condition essentially gives an upper bound on the maximum throughput, which is adjusted for the effect of the spillback.

The intuition of Proposition~\ref{Prp_Stable1} is as follows.
Suppose that we isolate link 2 from the two link system; i.e. consider a {\bf single} server, which we call link 2', with buffer size $\theta$ and a capacity switching between $u_1$ and $u_2$ with transition rates $\lambda$ and $\mu$. The queue length in link 2' is denoted by $q'(t)$.
Then, $\hat\sfp$ as defined in \eqref{Eq_phat} can be interpreted as the probability that the buffer is full Theorem 11.6 in \cite{kulkarni97}, i.e.
\begin{align*}
\lim_{t\to\infty}\frac1t\int_{\tau=0}^t\mathbf1_{q'(\tau)=\theta}d\tau
=\hat\sfp,\quad a.s.
\end{align*}
if a constant inflow $r$ is sent to link 2'.
In fact, $\hat\sfp$ is a lower bound for the probability that link 2 in the two-link system is full, provided that \eqref{Eq_Stable0} holds. The reason is that, under \eqref{Eq_Stable0}, the inflow sent to link 2 is no less than $r$ at all times.

In fact, \eqref{Eq_Stable0} is also a necessary condition for the stability of the two-link system. To see this, note that, if \eqref{Eq_Stable0} does not hold, then either the long-time average flow from link 1 to link 2 or the long-time average flow out of link 2 is less than the inflow $r$, which implies an unbounded queue.

The proof of the necessary condition is as follows:

\begin{proof}[Proof of Proposition \ref{Prp_Stable1}]
Suppose that the two-link system is stable.

From \eqref{Eq_D}, we obtain that
\begin{align*}
q_1(t)=\int_{\tau=0}^t\Big(r-s_1(\tau)\Big)d\tau+q_1(0),\quad t\ge0.
\end{align*}
We know that $\lim_{t\to\infty}q_1(0)/t=0$ for all $q_1(0)=q_{1}\ge0$. Thus, we have
\begin{align}
0&=\lim_{t\to\infty}\frac{1}{t}\Bigg(\int_{\tau=0}^t\Big(r-s_1(\tau)\Big)d\tau+q_1(0)-q_1(t)\Bigg)\nonumber\\
&=\lim_{t\to\infty}\frac{1}{t}\Bigg(\int_{\tau=0}^t\Big(r-s_1(\tau)\Big)d\tau-q_1(t)\Bigg).
\label{Eq_lim}
\end{align}

By \eqref{Eq_s1}, we see that
\begin{align*}
s_1(\tau)\le
\left\{\begin{array}{ll}
u_2, & i(\tau)=2,\;q_2(\tau)=\theta,\\
v, & o.w.,
\end{array}\right.
\quad\tau\ge0,
\end{align*}
which implies that
\begin{align}
&\lim_{t\to\infty}\frac{1}{t}\int_{\tau=0}^ts_1(\tau)d\tau\nonumber\\
&\le\lim_{t\to\infty}\frac{1}{t}\int_{\tau=0}^t\left(\mathbf1_{\substack{i(\tau)=2\\q_2(\tau)=\theta}}u_2+\left(1-\mathbf1_{\substack{i(\tau)=2\\q_2(\tau)=\theta}}\right)v\right)d\tau\nonumber\\
&=u_2\lim_{t\to\infty}\frac{1}{t}\int_{\tau=0}^t\mathbf1_{\substack{i(\tau)=2\\q_2(\tau)=\theta}}d\tau\nonumber\\
&\quad+v\left(1-\lim_{t\to\infty}\frac{1}{t}\int_{\tau=0}^t\mathbf1_{\substack{i(\tau)=2\\q_2(\tau)=\theta}}d\tau\right).
\label{Eq_lims1}
\end{align}
Since $s_1(\tau)\ge r$ for all $\tau\ge0$, the limiting fraction of time when link 2 is full is lower-bounded by the limiting fraction of time when link 2' is full; i.e.
\begin{align}
\lim_{t\to\infty}\frac{1}{t}\int_{\tau=0}^t\mathbf1_{\substack{i(\tau)=2\\q_2(\tau)=\theta}}d\tau\ge\hat\sfp.
\label{Eq_mathbb1}
\end{align}
Recalling from Assumption~\ref{Asm_uv} that $v>u_2$, we obtain from \eqref{Eq_lims1} and \eqref{Eq_mathbb1} that
\begin{align}
\lim_{t\to\infty}\frac{1}{t}\int_{\tau=0}^ts_1(\tau)d\tau\le \hat\sfp u_2+v(1-\hat\sfp).
\label{Eq_lims2}
\end{align}

In addition, note that, if the total queue length is bounded on average, then
\begin{align}
\lim_{t\to\infty}q_1(t)/t=0,\quad a.s.
\label{Eq_limq}
\end{align}

Combining \eqref{Eq_lim}, \eqref{Eq_lims2}, and \eqref{Eq_limq}, we obtain \eqref{Eq_Stable1}, completing the proof.
\end{proof}
\subsection{Sufficient condition: Foster-Lyapunov drift condition}
\label{Sec_Stable2}

Recall from Section~\ref{Sec_Model} the definition of inflow $r$, saturation rates $v$, $u_1$, and $u_2$, transition rates $\lambda$ and $\mu$, and buffer size $\theta$. The sufficient condition is as follows:

\begin{prp}
\label{Prp_Stable2}
Consider the two-link system defined in \eqref{Eq_System} and satisfying Assumption~\ref{Asm_uv}.
If there exist positive  constants $a_1$, $a_2$, $b_1$, $b_2$, $c$, and $d$ satisfying the linear inequalities
\begin{subequations}
\begin{align}
&2(r-v)+\lambda(b_2-b_1)\le-c,\label{Eq_LP1}\\
&2(r-v)+a_1(v-u_1)+\lambda(a_2-a_1)\theta+\lambda(b_2-b_1)\le-c,\label{Eq_LP2}\\
&2(r-v)+a_2(v-u_2)+\mu(b_1-b_2)\le-c,\label{Eq_LP3}\\
&2(r-v)+a_2(v-u_2)+\mu(a_1-a_2)\theta+\mu(b_1-b_2)\le-c,\label{Eq_LP3.5}\\
&2(r-u_2)+\mu(a_1-a_2)\theta+\mu(b_1-b_2)\le-c,\label{Eq_LP4}\\
&d\ge a_1(r-v)+c\theta,\label{Eq_LP5.1}\\
&d\ge a_2(r-u_2)+c\theta,\label{Eq_LP5.2}\\
&d\ge c\theta,\label{Eq_LP5.3}
\end{align}
\label{Eq_LP}
\end{subequations}
then, for any initial condition $(i,q)\in\I\times\Q$,
\begin{align}
\limsup_{t\to\infty}\frac1t\int_{\tau=0}^t\E[|q(\tau)|]d\tau\le d/c.
\label{Eq_UB}
\end{align}
\end{prp}

The above condition is easy to check, since the linear inequalities \eqref{Eq_LP} can be efficiently solved using known methods \cite{bertsimas97}. 

In general, there may be a gap between Proposition~\ref{Prp_Stable2} and the necessary condition, Proposition~\ref{Prp_Stable1} (see Section~\ref{Sec_Analysis} for examples).
Finally, as a sufficient condition, Proposition~\ref{Prp_Stable2} leads to lower bounds for the throughput of the two-link system.

Proposition~\ref{Prp_Stable2} is derived based on the Foster-Lyapunov drift condition Theorem 4.3 in \cite{meyn93}.
Here we recall this result for the sake of completeness:
if there exist a norm-like\footnote{According to \cite{meyn93}, a function $V:\I\times\Q\to\real_{\ge0}$ is norm-like if $\lim_{q_1\to\infty}V=\infty$ for all $i\in\I$.} function $V:\I\times\Q\to\real_{\ge0}$ (called the \emph{Lyapunov function}), a function $f:\I\times\Q\to[1,\infty)$, and constants $c>0$ and $d<\infty$ such that
\begin{align}
\Ell V(i,q)\le-cf(i,q)+d,\quad\forall(i,q)\in\I\times\Q,
\label{Eq_Drift}
\end{align}
then
\begin{align*}
\limsup_{t\to\infty}\frac1t\int_{\tau=0}^t\E\Big[f\Big(i(\tau),q(\tau)\Big)\Big]d\tau\le d/c.
\end{align*}
The main challenge for applying the Foster-Lyapunov drift condition is that the inequality \eqref{Eq_Drift} has to hold for all $(i,q)\in\I\times\Q$.
In this article, we propose a polynomial Lyapunov function, and employ properties of the fluid queueing dynamics to translate the drift condition to the form of linear inequalities.

In this article, we consider the following Lyapunov function:
\begin{align*}
V(i,q)=q_1^2+a_iq_1q_2+b_iq_1.
\end{align*}
This polynomial form is motivated by quadratic Lyapunov functions used for queueing networks \cite{kumar95}.
In addition, we consider
\begin{align*}
f(i,q)=q_1+q_2,\quad i=1,2.
\end{align*}
The proof of Proposition~\ref{Prp_Stable2} is as follows:

\begin{proof}[Proof of Proposition~\ref{Prp_Stable2}]
Applying the infinitesimal generator $\Ell$ to the Lyapunov function, we have
\begin{subequations}
\begin{align}
&\Ell V(1,q)=2q_1\dot q_1+a_1\dot q_1q_2+a_1q_1\dot q_2+b_1\dot q_1\nonumber\\
&\quad+\lambda(a_2-a_1)q_1q_2+\lambda(b_2-b_1)q_1\nonumber\\
&=(2(r-s_1(q_1,q_2))+a_1(s_1(q_1,q_2)-s_2(1,q_1,q_2))\nonumber\\
&\quad+\lambda(a_2-a_1)q_2+\lambda(b_2-b_1))q_1+a_1(r-s_1(q_1,q_2))q_2\label{Eq_LV1}\\
&\Ell V(2,q)=2q_1\dot q_1+a_2\dot q_1q_2+a_2q_1\dot q_2+b_2\dot q_1\nonumber\\
&\quad+\mu(a_1-a_2)q_1q_2+\mu(b_1-b_2)q_1\nonumber\\
&=(2(r-s_1(q_1,q_2))+a_2(s_1(q_1,q_2)-s_2(2,q_1,q_2))\nonumber\\
&\quad+\mu(a_1-a_2)q_2+\mu(b_1-b_2))q_1+a_2(r-s_1(q_1,q_2))q_2.\label{Eq_LV2}
\end{align}
\label{Eq_LV}
\end{subequations}
To check the drift condition \eqref{Eq_Drift}, we need to consider five cases:
\begin{enumerate}
\item $i=1$, $q_1>0$, and $q_2=0$. In this case, we have 
\begin{align*}
&\Ell V(1,q)
\stackrel{\eqref{Eq_LV1}}{=}(2(r-v)+a_1(v-v)+\lambda(b_2-b_1))q_1\\
&=(2(r-v)_\lambda(b_2-b_1))q_1
\stackrel{ \eqref{Eq_LP1}}{\le}-cq_1,
\end{align*}
which implies $\Ell V\le-c|q|+d$.

\item $i=1$, $q_1>0$, and $0<q_2\le\theta$. In this case, we have
\begin{align*}
&\Ell V(1,q)
\stackrel{\eqref{Eq_LV1}}{=}(2(r-v)+a_1(v-u_1)+\lambda(a_2-a_1)q_2\\
&\quad+\lambda(b_2-b_1))q_1+a_1(r-v)q_2\\
&\stackrel{ \eqref{Eq_LP2}}{\le}-cq_1+a_1(r-v)q_2
\stackrel{ \eqref{Eq_LP5.1}}{\le}-c(q_1+q_2)+d.
\end{align*}

\item $i=2$, $q_1>0$, and $0\le q_2<\theta$. In this case, we have \begin{align*}
&\Ell V(2,q)
\stackrel{\eqref{Eq_LV2}}{=}(2(r-v)+a_2(v-u_2)+\mu(a_1-a_2)q_2\\
&\quad+\mu(b_1-b_2))q_1+a_2(r-v)q_2\\
&\stackrel{\eqref{Eq_LP3}\eqref{Eq_LP3.5}}{\le}-cq_1+a_2(r-v)q_2
\stackrel{ \eqref{Eq_LP5.2}}{\le}-c(q_1+q_2)+d,
\end{align*}
where we have applied Assumption~\ref{Asm_uv} for the last inequality.

\item $i=2$, $q_1>0$, and $q_2=\theta$. In this case, we have
\begin{align*}
&\Ell V(2,q)
\stackrel{\eqref{Eq_LV2}}{=}(2(r-u_2)+a_2(u_2-u_2)+\mu(a_1-a_2)q_2\\
&\quad+\mu(b_1-b_2))q_1+a_2(r-u_2)q_2\\
&\stackrel{\eqref{Eq_LP3}\eqref{Eq_LP3.5}}{\le}-cq_1+a_2(r-u_2)q_2
\stackrel{ \eqref{Eq_LP5.2}}{\le}-c(q_1+q_2)+d,
\end{align*}
which implies $\Ell V\le-c|q|+d$.

\item $q_1=0$. In this case, we have
\begin{align*}
&\Ell V(i,q)
\stackrel{\eqref{Eq_LV}}{=}a_i(r-s(i,q_1,q_2))q_2
\stackrel{\eqref{Eq_s}}{\le} a_i(r-u_2)q_2\\
&\quad\stackrel{\eqref{Eq_LP5.1}\mbox{--}\eqref{Eq_LP5.3}}{\le}-cq_2+d,
\end{align*}
which implies $\Ell V\le-c|q|+d$.

\end{enumerate}
In conclusion, the drift condition \eqref{Eq_Drift} holds (for all $i\in\I$ and all $q\in\Q$). We can then obtain \eqref{Eq_UB} from Theorem 4.3 in \cite{meyn93}.

\end{proof}
\newpage
\section{Throughput analysis of two-link system}
\label{Sec_Analysis}

In this section, we use our results to study the impact due to capacity fluctuation and buffer size on the throughput of the two-link system.
Specifically, we consider the nominal (baseline) model specified by the following parameters:
\begin{subequations}
\begin{align}
&v=0.75,\;u_1=1,\;u_2=0.5,\\
&\lambda=1,\;\mu=1,\;\theta=1.
\end{align}
\label{Eq_Parameters}
\end{subequations}
We fix the average saturation rate of link 2 (i.e. $(u_1+u_2)=0.75$), and study how the throughput changes with (i) the magnitude of capacity variation, quantified by the quantity $\Delta u=u_1-u_2$, (ii) the frequency of capacity fluctuation, quantified by $\lambda$ and $\mu$, and (iii) the buffer size $\theta$.

Although it is not easy to compute the exact value of the maximum throughput $J_{\max}$, the necessary (resp. sufficient) condition leads to upper (resp lower) bounds for $J_{\max}$. For a given two-link system, if an inflow value $r_1$ does not satisfy Proposition~\ref{Prp_Stable1}, then the system is unstable with the inflow $r_1$, and we can thus conclude that $r_1$ an upper bound for $J_{\max}$.
Similarly, if an inflow value $r_2$ satisfies Proposition~\ref{Prp_Stable2}, then the system is stable with the inflow $r_2$, and we can thus conclude that $r_2$ a lower bound for $J_{\max}$.
The gap between the bounds can be narrowed by minimize (resp. maximize) the upper (resp. lower) bound.

\subsection{Magnitude of capacity variation}

Suppose that $u_1$ and $u_2$ are specified as follows:
\begin{align*}
u_1=0.75+\Delta u/2,\;
u_2=0.75-\Delta u/2,\;\Delta u\in[0,1.5].
\end{align*}
Note that the upper bound for $\Delta u$ ensures that $u_2$ is non-negative.
For various values of $\Delta u$, we use Propositions~\ref{Prp_Stable1} and \ref{Prp_Stable2} to obtain upper and lower bounds for $J_{\max}$; we also numerically optimize the bounds. The results are plotted in Figure~\ref{Fig_J_Deltau}.

\begin{figure}[hbt]
\centering
\subfigure[Throughput versus capacity variation.]{
\centering
\includegraphics[width=0.31\textwidth]{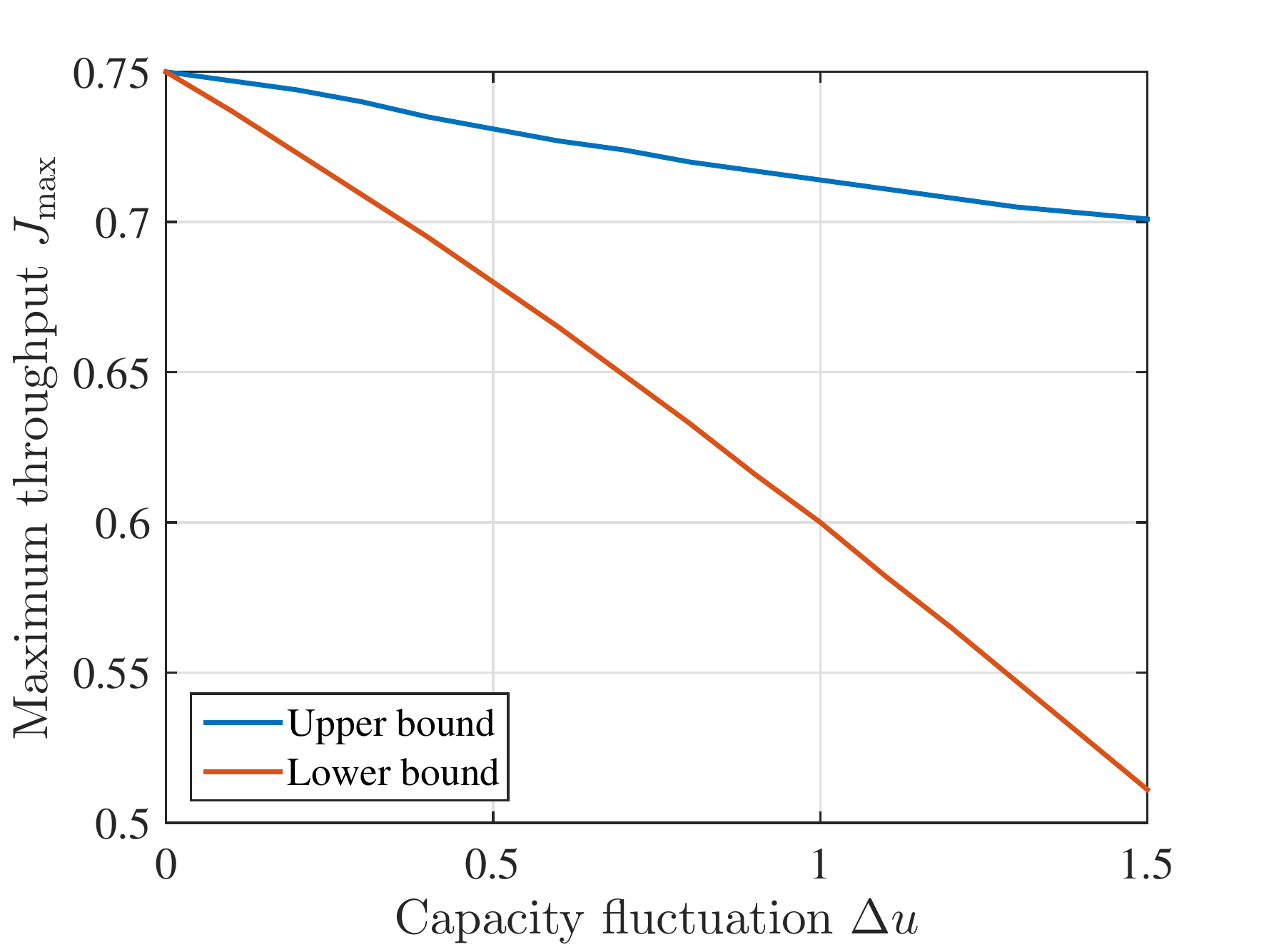}
\label{Fig_J_Deltau}
}
\subfigure[Throughput versus (logarithmized) transition rates.]{
\centering
\includegraphics[width=0.31\textwidth]{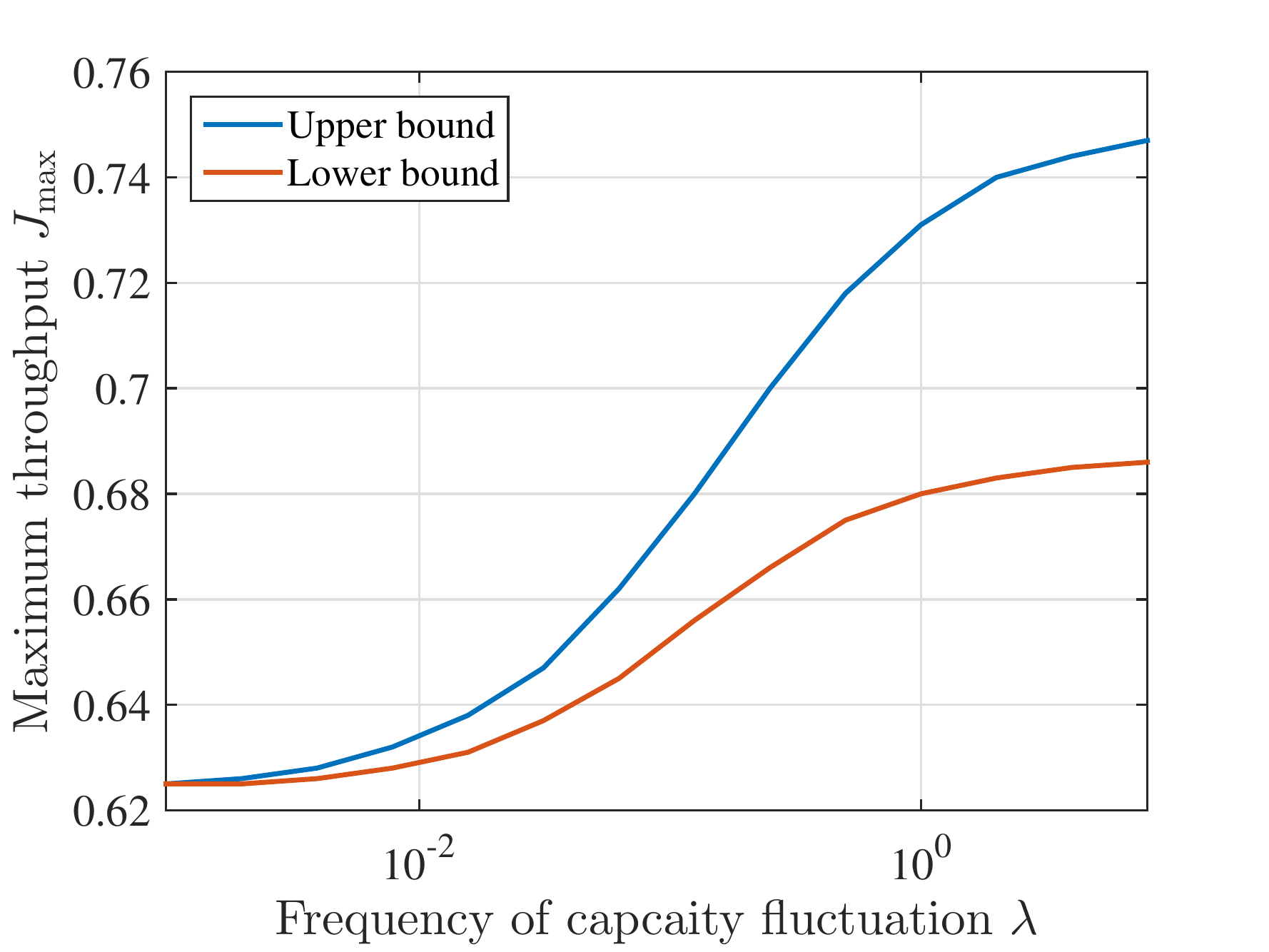}
\label{Fig_J_lambda}
}
\subfigure[Throughput versus (logarithmized) buffer size.]{
\centering
\includegraphics[width=0.31\textwidth]{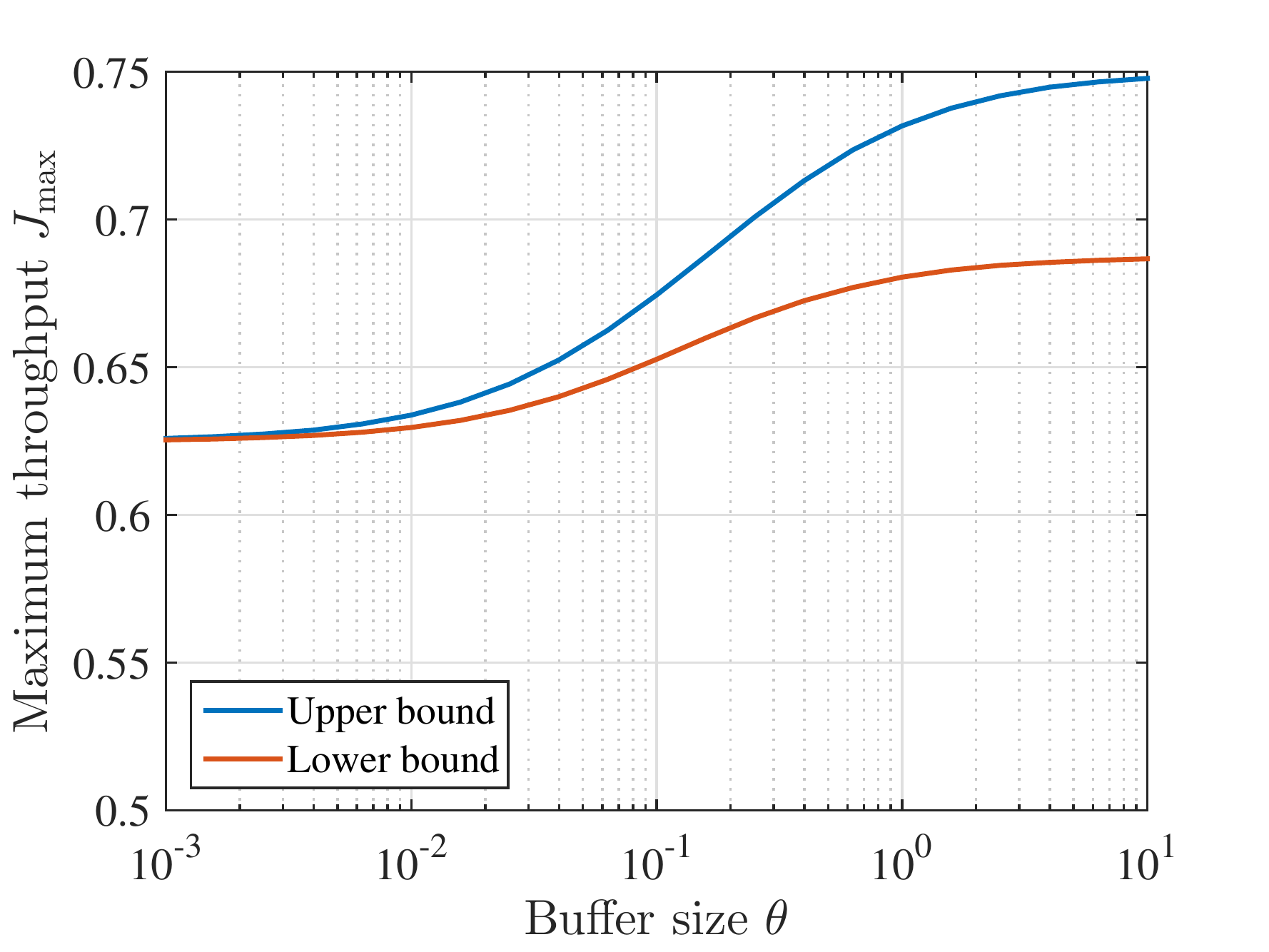}
\label{Fig_J_theta}
}
\caption{Throughput versus various model parameters.}
\label{Fig_J}
\end{figure}

The results imply that, with the average capacity fixed, the maximum throughput decreases as the magnitude of capacity fluctuation increases. In addition, both bounds converges to 0.75 as $\Delta u$ approaches 0; this is intuitive in that link 2 reduces to a server with a deterministic capacity of 0.75 as $\Delta u$ approaches 0. Finally, as $\Delta u$ increases, the gap between the upper and the lower bounds increases.

From a practical viewpoint, the actual throughput of a transportation facility can be strictly less than its average capacity. The reason for this phenomenon is that capacity fluctuation leads to spillback and queues at the upstream sections or stages, which in turn leads to additional bottlenecks at those sections or stages; without capacity fluctuation, the additional bottlenecks do not arise. Therefore, capacity fluctuation plays a very important role in throughput analysis.

\subsection{Frequency of capacity fluctuation}
Now consider the system with $u_1=1$ and $u_2=0.5$, but with $\lambda$ varying. To fix the average capacity at 0.75, we always set $\mu=\lambda$ as $\lambda$ varies.
For various values of $\lambda$, we use Propositions~\ref{Prp_Stable1} and \ref{Prp_Stable2} to obtain upper and lower bounds for $J_{\max}$; we also numerically optimize the bounds. The results are plotted in Figure~\ref{Fig_J_lambda}.

The results imply that the maximum throughput increases as the frequency of capacity fluctuation $\theta$ increases. Both the upper and the lower bounds converges to 0.625 as $\lambda$ (and $\mu$) approaches 0. 
The reason is that, as $\lambda$ and $\mu$ approaches 0, the time intervals between mode transitions are very long on average; consequently, link 2 is empty during most of the time when $i(t)=1$ and is full during most of the time when $i(t)=2$.
Thus, the behavior of the two-link system is similar to a single link with a capacity switching between 0.75 and 0.5;

From a practical viewpoint, note that $\lambda$ characterizes the frequency of capacity disruptions and $\mu$ characterizes the duration of capacity disruptions. Hence, our result implies that less frequent but longer-lasting capacity disruption leads to larger throughput loss than more frequent but shorter-lasting capacity disruptions. 
The reason for this phenomenon is that, with the buffer size fixed, longer capacity disruptions are more likely to cause queue spillback than shorter ones.
In other words, frequent but short disruptions are not likely to cause congestion sufficiently severe to lead to spillback.

\subsection{Buffer size}

Now consider the system with $u_1=1$, $u_2=0.5$, $\lambda=\mu=1$, but with $\theta$ varying.
For various values of $\theta$, we use Propositions~\ref{Prp_Stable1} and \ref{Prp_Stable2} to obtain upper and lower bounds for $J_{\max}$; we also numerically optimize the bounds. The results are plotted in Figure~\ref{Fig_J_theta}.

The results imply that the maximum throughput increases as the buffer size $\theta$ increases. Both the upper and the lower bounds converges to 0.625 as $\theta$ approaches 0. This is intuitive: when the buffer size of link 2 is very small, the two-link system can be in fact viewed as a single link with a capacity switching between $v$ and $u_2$, and the average capacity is close to $0.5v+0.5u_2=0.625$. As $\theta$ increases, the gap between the upper and the lower bounds increases.

From a practical viewpoint, the throughput of a transportation facility not only depends on its capability of discharging queues, but also how much traffic it can store. If its storage space is limited and frequently leads to queues at the upstream sections or stages, then additional bottlenecks can be produced, which undermines the throughput.

In addition, comparison of Figures~\ref{Fig_J_Deltau}--\ref{Fig_J_theta} implies that throughput is more sensitive to capacity variation than to frequency of capacity fluctuation and to buffer size. To see this, recall that the baseline model has a capacity variation of 0.5, transition rates 1, and a buffer size of 1. If the capacity variation is doubled from 0.5 to 1, the upper (resp. lower) bound is decreased by 13\% (resp. 2.3\%). 
If the frequency of capacity fluctuation is doubled from 1 to 2, the upper (resp. lower) bound is increased by 1.2\% (resp. 0.4\%).
If the buffer size is doubled from 1 to 2, the upper (resp. lower) bound is increased by 0.6\% (resp. 1.1\%).
\newpage
\section{Further discussion}
\label{Sec_Discussion}

In this section, we discuss two possible extensions of our results on two-link systems with constant inflows.
In Section~\ref{Sec_Control}, we discuss the extension of the stability conditions to the cases where the inflow is specified by a mode-responsive control policy instead of constant, and compare some properties of the controlled system with those of the uncontrolled system.
In Section~\ref{Sec_Net}, we discuss the extension to merges and splits, which are the basic structures in general networks.

\subsection{Feedback-controlled system}
\label{Sec_Control}

Suppose that, instead of being constant, the inflow $r$ is specified by a function $\phi:\I\to\R$ such that
\begin{align}
r(t)=\phi(i(t))
=\left\{\begin{array}{ll}
r_1, & i(t)=1,\\
r_2, & i(t)=2.
\end{array}\right.
\label{Eq_phi}
\end{align}
Our stability conditions can be easily extended to feedback-controlled system as follows.

Define
\begin{align*}
\tilde D:=\left[\begin{array}{cc}
\min\{r_1,v\}-u_1 & 0\\
0 & \min\{r_2,v\}-u_2
\end{array}\right]
\end{align*}
and let $\Lambda$ as defined in \eqref{Eq_Lmd}.
Let $\tilde w_1$ and $\tilde w_2$ be the distinct solutions to the equation
\begin{align*}
\det[w\tilde D-\Lambda]=0.
\end{align*}
Let row vectors $\tilde \phi_1=[\tilde \phi_{11},\tilde \phi_{12}]$ and $\tilde \phi_2=[\tilde \phi_{21},\tilde \phi_{22}]$ and scalars $\tilde k_1$, $\tilde k_2$ be the solutions to
\begin{align*}
&\tilde \phi_1[\tilde w_1\tilde D-\Lambda]=0,\\
&\tilde \phi_2[\tilde w_2\tilde D-\Lambda]=0,\\
&\tilde k_1\tilde \phi_{12}+\tilde k_2\tilde \phi_{22}=0,\\
&\tilde k_1\tilde \phi_{11}\exp(\tilde w_1\theta)+\tilde k_2\tilde \phi_{21}\exp(\tilde w_2\theta)=\frac{\mu}{\lambda+\mu}.
\end{align*}
Next, define
\begin{align}
\tilde \sfp:=\frac{\lambda}{\lambda+\mu}- \tilde k_1\tilde \phi_{12}\exp(\tilde w_1\theta)-\tilde k_2\tilde \phi_{22}\exp(\tilde w_2\theta).
\label{Eq_ptilde}
\end{align}
Then, recall the definition of $v$, we state the necessary condition as follows:

\begin{prp}
\label{Prp_Stable3}
Consider the two-link system with the control policy $\phi$ given by \eqref{Eq_phi}.
If the total queue lengths are bounded on average in the sense of \eqref{Eq_Bounded}, then
\begin{align}
\sfp_1\min\{r_1,v\}+\sfp_2\min\{r_2,v\}\le(1-\tilde \sfp)v+\tilde \sfp u_2,
\label{Eq_Stable1}
\end{align}
where $\tilde \sfp$ is defined in \eqref{Eq_ptilde}.
\end{prp}

In addition, a sufficient condition for stability of the controlled system is as follows:

\begin{prp}
\label{Prp_Stable4}
Consider the two-link system defined in \eqref{Eq_System}.
If there exist positive  constants $a_1$, $a_2$, $b_1$, $b_2$, $c$, and $d$ satisfying the linear inequalities
\begin{align*}
&2(r_1-v)+\lambda(b_2-b_1)\le-c,\\
&2(r_1-v)+a_1(v-u_1)+\lambda(a_2-a_1)\theta+\lambda(b_2-b_1)\le-c,\\
&2(r_2-v)+a_2(v-u_2)+\mu(b_1-b_2)\le-c,\\
&2(r_2-v)+a_2(v-u_2)+\mu(a_1-a_2)\theta+\mu(b_1-b_2)\le-c,\\
&2(r_2-u_2)+\mu(a_1-a_2)\theta+\mu(b_1-b_2)\le-c,\\
&d\ge a_1(r_1-v)+c\theta,\\
&d\ge a_2(r_2-u_2)+c\theta,\\
&d\ge c\theta,
\end{align*}
then, for any initial condition $(i,q)\in\I\times\Q$,
\begin{align*}
\limsup_{t\to\infty}\frac1t\int_{\tau=0}^t\E[|q(\tau)|]d\tau\le d/c.
\end{align*}
\end{prp}

We now compare the open-loop system and the closed-loop system. 

First, the closed-loop system can achieve better performance in the sense of smaller queues. 
Consider a two-link system with a constant inflow $r=0.625$ (called $S_1$), and a two-link system with a mode-responsive control policy such that $r_1=0.75$ and $r_2=0.5$ (called $S_2$).
Both $S_1$ and $S_2$ have parameters given in \eqref{Eq_Parameters}.
By Propositions~\ref{Prp_Stable2} and \ref{Prp_Stable4}, both $S_1$ and $S_2$ are stable. However, the long-time average queue length in $S_1$ is positive, while that in $S_2$ is zero.
Hence, mode-responsive control reduces the queueing delay in the two-link system.

Second, the performance of $S_2$ depends on the estimate of the mode.
In reality, the mode can be observed from real-time measurement of traffic condition or surveillance of traffic incidents.
If the estimate of the mode is accurate, then $S_2$ has a smaller queueing delay. However, if the measurement of the mode is inaccurate, then the performance of $S_2$ is not necessarily better than $S_1$. For example, if the sensor fails for some reason and never reports capacity reduction, then the system operator finds the system to be in mode 1 for all time and thus sends a constant inflow of 0.75 to the system; by Proposition~\ref{Prp_Stable1}, this decision leads to instability.
Such vulnerability does not exist in $S_1$.
Development of practically relevant models for reliability/security sensor failures \cite{hoh08,laszka16} is part of our ongoing work.
\subsection{Extension to merges and splits}
\label{Sec_Net}

Now we discuss how our approach can be extended from the two-link system to merges and splits.
This discussion is helpful to understand the behavior of more general fluid queueing networks with finite buffer sizes.

\begin{figure}[hbt]
\centering
\subfigure[A merge.]{
\centering
\includegraphics[width=0.31\textwidth]{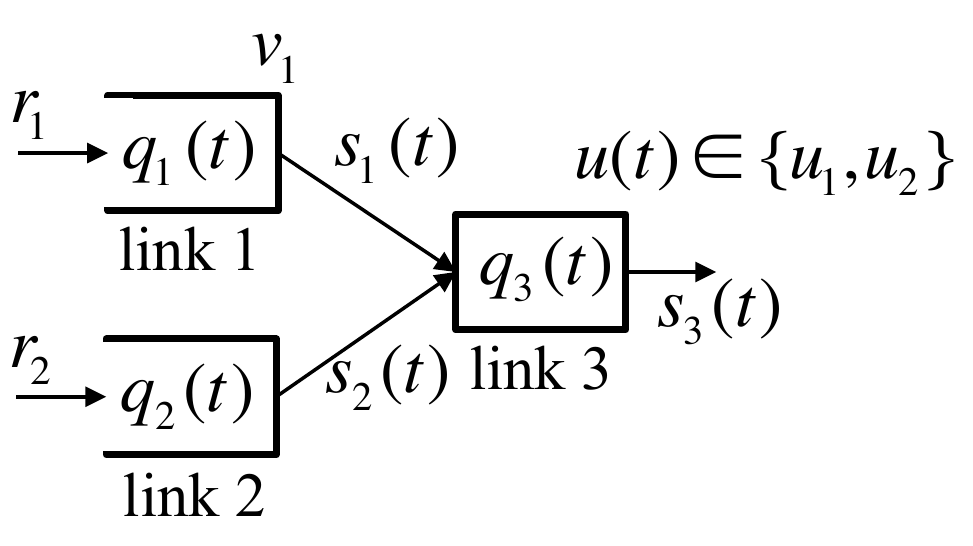}
\label{Fig_Merge}
}
\subfigure[A split.]{
\centering
\includegraphics[width=0.31\textwidth]{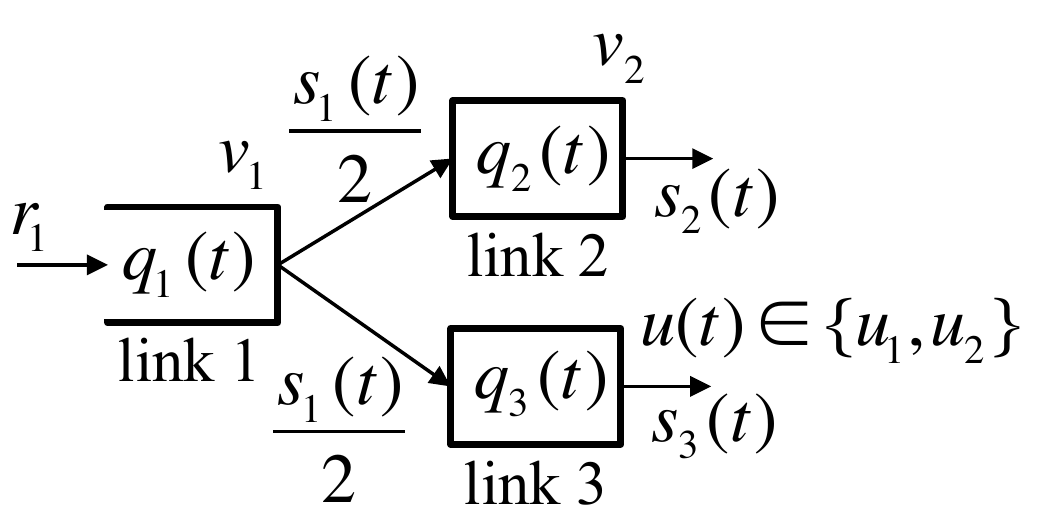}
\label{Fig_Split}
}
\caption{Tandem fluid queuing models for merges/splits.}
\end{figure}

\subsubsection{Merges}
Consider a merge, i.e. three fluid queueing links connected as in Figure~\ref{Fig_Merge}. Links 1 and 2 have constant capacities and infinite buffer size. Link 3 has a capacity switching between $u_1$ and $u_2$, and a buffer size of $\theta$.
The state space is $\Q_m=[0,\infty)^2\times[0,\theta]$.
Furthermore, we assume that flow from link 1 is prioritized over link 2.

The necessary condition can be extended by considering link 3 isolated from the merge system, and compute a lower bound for spillback probability.

To extend the sufficient condition, we consider the following Lyapunov function:
\begin{align*}
V_m(i,q)=q_1^2+a_{1,i}q_1q_3+b_{1,i}q_1
+q_2^2+a_{2,i}q_2q_3+b_{2,i}q_2,
\end{align*}
where $a_{1,i}$, $a_{2,i}$, $b_{1,i}$, and $b_{2,i}$ are positive constants.
Then, if one can find positive constants $c$ and $d$ such that
\begin{align}
\Ell V_m(i,q)\le-c(q_1+q_2+q_3)+d,
\label{Eq_Drift2}
\end{align}
then the merge system is stable. Similar to the case of the single-link system, \eqref{Eq_Drift2} can be translated to a set of linear inequalities, which are not hard to solve.

The behavior of a merge is similar to that of the two-link system. However, an important distinct property of a merge is that stability not only depends on the sum of the inflows, but also how the inflows are distributed over the upstream links.
To see this, suppose that
\begin{align*}
r_{(1)}+r_{(2)}=R<u_2,\;
r_{(1)}>v_1,
\end{align*}
for some $R>0$.
Then, the system is unstable in that $q_1$ grows unboundedly.
However, suppose that
\begin{align*}
r_{(1)}+r_{(2)}=R<u_2,\;
r_{(1)}\le v_1,\;
r_{(2)}\le v_2.
\end{align*}
Then, although the total inflow is unchanged, the system becomes stable.


\subsubsection{Splits}
Consider a split, i.e. three fluid queueing links connected as in Figure~\ref{Fig_Merge}. For ease of presentation, we assume that links 1 and 2 have constant capacities and infinite buffer size, while link 3 has a capacity switching between $u_1$ and $u_2$, and a buffer size of $\theta$.
We assume that outflow from link 1 is evenly distributed to links 2 and 3; i.e. $r/2$ amount of traffic is assigned to the route consisting of links 1 and 2 (resp. 3), which we call route $\{1,2\}$ (resp. route $\{1,3\}$) for short).

The necessary condition can be extended by considering link 3 isolated from the merge system, and compute a lower bound for spillback probability.

To extend the sufficient condition, we consider the following Lyapunov function:
\begin{align*}
V_s(i,q)=q_1^2+(a_{i}q_2+a_i'q_3)q_1+b_{i}q_1
\end{align*}
where $a_{i}$, $a_{i}'$, and $b_{i}$ are positive constants.
Then, if one can find positive constants $c$ and $d$ such that
\begin{align}
\Ell V_s(i,q)\le-c(q_1+q_2+q_3)+d,
\label{Eq_Drift3}
\end{align}
then the split system is stable. Similar to the case of the single-link system, \eqref{Eq_Drift3} can be translated to a set of linear inequalities, which are not hard to solve.

The most important property of a split is that congestion in one downstream link may block traffic into the other downstream link.
To see this, suppose that
\begin{align*}
&r=1.6,\;v_1=2,\;v_2=1,\\
&u_1=1,\;u_2=0.5,\;
\lambda=\mu=1,\;\theta=0.
\end{align*}
That is, 0.8 amount of traffic is assigned to each of routes $\{1,2\}$ and $\{1,3\}$, which is strictly less than the nominal/average capacity of link 2 and link 3, respectively.
However, the system is unstable.
To see this, note that, for 50\% of the time, the split experiences spillback. During spillback, we have $s_1=2u_2=2(0.5)=1$. Hence, the split system is unstable in that $q_1$ grows unboundedly a.s.
\newpage
\section{Concluding Remarks}
\label{Sec_Conclude}

In this article, we present an analysis of a two-link fluid queueing system with both stochastic capacity and spillback. We derive a necessary condition and a sufficient condition for the stability of the two-link system. The necessary condition implies that the two-link system is not necessarily stable even if the inflow is strictly less than the average capacity of each link. The sufficient conditions provide stability guarantee to a set of inflow values. Using these results, we analyze the throughput of the two-link system. We also discuss how our analysis can be extended to feedback-controlled systems, and to more general networks.

\newpage
\section*{Acknowledgment}
This work was supported by the Singapore NRF Future Urban Mobility project, NSF CNS-1239054 CPS Frontiers: \underline{F}oundations \underline{O}f \underline{R}esilient \underline{C}yb\underline{E}r-Physical \underline{S}ystems (FORCES), NSF CAREER Award CNS-1453126, AFRL Lablet-Secure and Resilient Cyber-Physical Systems, and NYU Tandon School of Engineering start-up funding.
We also appreciate the valuable inputs from the reviewers.
\newpage

\bibliographystyle{trb}
\bibliography{Bibliography}


\end{document}